\documentclass[a4paper,reqno,12pt]{amsart}

\usepackage{amsfonts,amsmath,amssymb,mathrsfs,mathtools,enumerate,bm,bbm,bbold,comment,tikz,tikz-cd,color}
\usepackage[colorlinks]{hyperref}
\usepackage[all]{xy}
\usepackage[normalem]{ulem}
\usepackage[margin = 1.25in, centering]{geometry}

\allowdisplaybreaks
\numberwithin{equation}{section}

\newtheorem{theorem}{Theorem}[section]
\newtheorem{proposition}[theorem]{Proposition}
\newtheorem{lemma}[theorem]{Lemma}
\newtheorem{corollary}[theorem]{Corollary}

\theoremstyle{definition}
\newtheorem{definition}[theorem]{Definition}

\newtheorem{question}[theorem]{Question}
\theoremstyle{remark}
\newtheorem{remark}[theorem]{Remark}
\newtheorem{example}[theorem]{Example}

\newcommand{\G}{\mathcal{G}}
\newcommand{\src}{\mathbf{d}}
\newcommand{\so}{\mathbf{s}}
\newcommand{\ran}{\mathbf{r}}
\newcommand{\one}{\mathbf{1}}

\newcommand{\supp}{\operatorname{supp}}

\newcommand{\pth}{\operatorname{Path}}

\date{\today}

\title[Maximal commutative subalgebras of Steinberg algebras]{Maximal commutative subalgebras of Steinberg algebras}

\author[A.\ Cichocka]{Anna Cichocka}
\author[Z.\ Mesyan]{Zachary Mesyan}
\author[M.\ Ziembowski]{Micha\l{} Ziembowski}

\address{Anna Cichocka and Micha\l{} Ziembowski: Faculty of Mathematics and Information Science, Technical University of Warsaw, 00-661 Warsaw, Poland}
\email{a.cichocka@mini.pw.edu.pl}
\email{m.ziembowski@mini.pw.edu.pl}
\address{Zachary Mesyan: Department of Mathematics, University of Colorado, Colorado Springs, CO 80918, USA}
\email{zmesyan@uccs.edu}



\begin{document}

\begin{abstract}
We construct large classes of maximal commutative subalgebras in prime Steinberg algebras, generalizing a known result for Leavitt path algebras.
\end{abstract}

\maketitle

\section{Introduction} 

There is a long history of studying maximal commutative subalgebras of various algebras, as a means of understanding their structure. For example, in a finite-dimensional semisimple Lie algebra over a complex number field, maximal commutative subalgebras are precisely the Cartan subalgebras~\cite{dynkin, malcev}. For another example, according to a classical theorem of Jacobson~\cite{jacobson}, which generalizes an earlier result of Schur~\cite{schur} for algebraically closed fields, a commutative subalgebra of a matrix ring $\mathbb{M}_n(K)$, over a field $K$, must have dimension at most $\lfloor n^2/4 \rfloor + 1$. Moreover, this bound is sharp (see~\cite{SBWZ} for more details). Among other settings, maximal commutative algebras have also been examined in Weyl algebras~\cite{bavula}, Leavitt path algebras~\cite{BvWZ_prime, BCvWZ}, Kumjian--Pask algebras~\cite{clark-canto-isfahani}, associative superalgebras~\cite{ELS}, $C^*$-algebras~\cite{KW,karamzadeh}, and Malcev algebras~\cite{elduque}.

In this note we construct maximal commutative subalgebras in certain Steinberg algebras. These algebras were first introduced by Steinberg~\cite{steinberg0} and Orloﬀ Clark/Farthing/Sims/Tomforde~\cite{SteinbergIntro}. They are algebraic analogues of groupoid $C^*$-algebras, developed by Renault~\cite{renault}, and constitute a very general class of rings, which includes groupoid algebras, inverse semigroup algebras, and Leavitt path algebras (which, in turn, include matrix and Laurent polynomial algebras~\cite{LPAbook}). Since their introduction, Steinberg algebras have received considerable attention in the literature. Most relevantly to the present context, Hazrat and Li~\cite{Hazrat-Li} constructed maximal commutative subalgebras of Steinberg algebras whose isotropy groups have abelian interiors. See~\cite{lisarooz} for an overview and more details about the history of Steinberg algebras.

Given a (discrete) commutative unital ring $R$ and a topological groupoid $\G$, with certain additional properties, a Steinberg algebra $A_R(\G)$ consists of continuous functions $\G \to R$ with compact support. (See Section~\ref{prelim-section} for more details.) Our main result (Theorem~\ref{steinbergtheorem}) is that if $A_R(\G)$ is prime, and $U_1 \cup U_2$ is a partition of the unit space (set of objects) of $\G$ into open sets, then
\[
\mathcal{Z}(A_R(G)) + \{f \in A_R(\G) \mid f(x) = 0 \text{ for all } x \notin U_2\G U_1 \}
\]
is a maximal commutative subalgebra of $A_R(\G)$, where $\mathcal{Z}(A_R(G))$ denotes the center. Specializing this to Leavitt path algebras (Theorem~\ref{LPAtheorem} and Corollary~\ref{LPA-cor}) gives more general versions of the main result of~\cite{BvWZ_prime}. Along the way, we show that a prime Steinberg algebra $A_R(\G)$ has a trivial center, whenever $\G$ is a groupoid whose unit space is not compact and has at least two elements (Proposition~\ref{steinberg-center}).

\section{Preliminaries} \label{prelim-section}

We begin by recalling basic concepts and results pertaining to Steinberg algebras.

\subsection{Groupoids}

Let $\G$ be a \emph{groupoid}, i.e., a small category in which every morphism is invertible. We denote the set of objects of $\G$, also known as the \textit{unit space}, by $\G^{(0)}$, and we identify these objects with the corresponding identity morphisms. For each morphism $x \in \G$, the object $\src(x):=x^{-1}x$ is the \textit{domain} of $x$, and $\ran(x):=xx^{-1}$ is its \emph{range}. Two morphisms $x$ and $y$ are composable as $xy$ if and only if $\src(x)=\ran(y)$. For all $X,Y \subseteq \G$, we define
\begin{equation}\label{setmult}
XY :=\big \{xy \mid x \in X, y \in Y \text{ such that } \src(x)=\ran(y) \big\},
\end{equation}
and
\begin{equation}\label{setinv}
X^{-1} :=\big \{x^{-1} \mid x \in X \big\}.
\end{equation}

A \textit{topological groupoid} is a groupoid (whose set of morphisms is) equipped with a topology making inversion and composition continuous. A topological groupoid $\G$ is \textit{\'etale} if $\G^{(0)}$ is locally compact and Hausdorff in the topology induced by that on $\G$, and $\src: \G \to \G^{(0)}$ is a local homeomorphism. In this situation $\ran: \G \to \G^{(0)}$ is necessarily also a local homeomorphism~\cite[p.\ 29]{rigby}. An \'etale groupoid $\G$ is \textit{topologically transitive} if $\src^{-1}(U)\cap \ran^{-1}(V) \neq \emptyset$ for all nonempty open $U, V \subseteq \G^{(0)}$. 

An open subset $X$ of an \'etale groupoid $\G$ is called a \textit{slice} or an \textit{open bisection} if the restrictions $\src|_{X}$ and $\ran|_{X}$ are injective (and hence are homeomorphisms onto their images). The collection of all slices forms a base for the topology of an \'etale groupoid $\G$, and $\G^{(0)}$ is a slice which is also closed~\cite[p.\  29]{rigby}. An \'etale groupoid is \textit{ample} if the compact slices form a base for its topology. The set $\G^{co}$ of all compact slices of an ample groupoid $\G$ is an inverse semigroup, under the operations given in equations (\ref{setmult}) and (\ref{setinv}) \cite[Proposition 2.2]{rigby}. (Recall that a semigroup $S$ is an \textit{inverse semigroup} if for each $s \in S$ there is a unique $s^{-1} \in S$ satisfying $s = ss^{-1}s$ and $s^{-1} = s^{-1}ss^{-1}$.)

\subsection{Steinberg algebras}

Let $\G$ be a Hausdorff ample groupoid, let $R$ be a commutative unital ring, and let $R^{\G}$ denote the set of all continuous functions $\G \to R$, where $R$ is viewed as a discrete topological space. For each $f \in R^{\G}$, let $\supp(f) := \{x \in \G \mid f(x) \neq 0\}$, and let 
\[A_R(\G): = \{f \in R^{\G} \mid \supp(f) \text{ is compact}\}.\]
For all $f,g \in A_R(\G)$ and $z \in \G$, define the \textit{convolution} by
\[f*g(z) := \sum_{xy = z}f(x)g(y).\]
Using the fact that $\G$ is \'etale, one can show that this sum is always finite, and so the operation is well-defined--see~\cite[Proposition 2.4]{rigby} or~\cite[Proposition 4.5]{steinberg0}. With pointwise addition and scalar multiplication, and convolution as multiplication, $A_R(\G)$ becomes an $R$-algebra, called a \textit{Steinberg algebra}.

For each $B \in \G^{co}$, let $\one_B : \G \to \{0,1\}$ denote the \textit{characteristic} function of $B$ (i.e., $\one_B(x) = 1$ if $x \in B$, and $\one_B(x) = 0$ otherwise). Since $B$ is compact, $\one_B \in A_R(\G)$. It can be shown that $A_R(\G)$ is spanned by $\{\one_B \mid B \in \G^{co}\}$ as an $R$-algebra~\cite[Corollary 2.2]{rigby}. In particular, $\supp(f)$ is open for any $f\in A_R(\G)$. It is also easy to see that $(\one_B*\one_D =) \, \one_B\one_D = \one_{BD}$, for all $B,D \in \G^{co}$. More generally, for all $x \in \G$, $f \in A_R(\G)$, and $B \in \G^{co}$, we have
\begin{equation} \label{setofeq1}
f\one_B(x) = \begin{cases} f(xy^{-1}) & \text{if } \src(y) = \src(x) \text{ for some (unique) } y\in B \\ 0 & \text{if } \src(x) \not\in\src(B) \end{cases}
\end{equation}
and
\begin{equation} \label{setofeq2}
\one_Bf(x) = \begin{cases} f(y^{-1}x) & \text{if } \ran(y) = \ran(x) \text{ for some (unique) } y\in B \\0 & \text{if } \ran(x) \not\in\ran(B). \end{cases}
\end{equation}
We refer the reader to~\cite{rigby} or~\cite{steinberg0} for additional details about Steinberg algebras. We note that in those publications Steinberg algebras are defined over ample groupoids that are not necessarily Hausdorff, which requires a slightly more complicated approach than the one presented here.

\section{Main results}

Throughout this section we focus on prime Steinberg algebras. These are described, in various situations, in~\cite{steinberg1}. In particular, according to~\cite[Proposition 4.3]{steinberg1}, if $A_R(\G)$ is prime, then the (commutative unital) ring $R$ must actually be an integral domain. For this reason we consider only base rings of this sort in most of our results.

The next definition gives the main ingredients of our construction of a maximal commutative subalgebra of $A_R(\G)$.

\begin{definition} \label{partitiondef}
Let $\G$ be a Hausdorff ample groupoid. Given any disjoint open $U_1, U_2 \subseteq \G^{(0)}$ such that $\G^{(0)} = U_1 \cup U_2$, we define
\begin{align*}
U_{12} & := U_1 \cap \ran(\src^{-1}(U_2)), & U_{11} & := U_1\setminus U_{12}, \\ 
U_{21} & := U_2 \cap \ran(\src^{-1}(U_1)), & U_{22} & := U_2 \setminus U_{21}.
\end{align*}
Additionally, given any commutative unital ring $R$ define
\[
A_{ij}  := \{f \in A_R(\G) \mid \supp(f) \subseteq U_i\G U_j \},
\]
for all $i,j \in \{1,2\}$. 
\end{definition}

We shall show that if $A_R(\G)$ is a prime Steinberg algebra, then $\mathcal{Z}(A_R(\G)) + A_{21}$ is a maximal commutative subalgebra, where $\mathcal{Z}(A_R(\G))$ denotes the center of $A_R(\G)$. The main challenge will be to demonstrate that certain elements of the centralizer of $A_{21}$ belong to $\mathcal{Z}(A_R(\G))$. Toward that end, our first goal is to show that an element of a prime $A_R(\G)$ belongs to $\mathcal{Z}(A_R(\G))$ if and only if it behaves ``well" on $\src^{-1}(U_{12} \cup U_{21})$, as defined above. We begin with a lemma.

\begin{lemma} \label{INT_U}
Let $\G$ be a Hausdorff ample groupoid, and $R$ an integral domain, such that $A_R(\G)$ is prime. Suppose that $U_{i}$ and $U_{ij}$ $(i,j \in \{1,2\})$ are as in Definition~\ref{partitiondef}, and let $V := U_{11} \cup U_{22}$ and $W := U_{12} \cup U_{21}$. Then the following hold.
\begin{enumerate}[\rm (1)]
\item $V^{\circ} = \emptyset$, where $V^{\circ}$ denotes the interior of $V$.

\item Let $x\in \src^{-1}(W)$ and $y\in \G$ be such that $\src(x) = \ran(y)$. Then $xy \in \src^{-1}(W)$. 
\end{enumerate}
\end{lemma}

\begin{proof}
(1) First, we note that since $\src$ is continuous and $\ran$ is open, $U_{12}$ and $U_{21}$ are open subsets of $\G^{(0)}$. Since $A_R(\G)$ is prime, $\G$ is topologically transitive, by~\cite[Proposition 4.3]{steinberg1}, and so $U_{12} \neq \emptyset$ and $U_{21} \neq \emptyset$.

Now, suppose that $V^{\circ} \neq \emptyset$. Since $U_1$ and $U_2$ are open, and $\G^{(0)} = U_1 \cup U_2$, it must be the case that either $(V \cap U_1)^{\circ} \neq \emptyset$ or $(V \cap U_2)^{\circ} \neq \emptyset$. Let us suppose that $(V \cap U_1)^{\circ} \neq \emptyset$, since the other case can be handled analogously. Then $(U_{11})^{\circ}\neq \emptyset$, and so there exists a nonempty $B_1 \in \G^{co}$ such that $B_1 \subseteq U_{11}$, since the elements of $\G^{co}$ form a base for the topology on $\G$. By the previous paragraph, we can find a nonempty $B_2 \in \G^{co}$ such that $B_2 \subseteq U_{21}$. Since $\G$ is topologically transitive, there must then exist $x \in \G$ such that $\ran(x)\in B_1 \subseteq U_{11}$ and $\src(x) \in B_2 \subseteq U_{2}$, contradicting the definition of $U_{11}$. Thus $V^{\circ} = \emptyset$.

\smallskip

(2) Let us assume that $\src(x)\in U_{12}$, since the case where $\src(x)\in U_{21}$ can be proved similarly.
    
If $\src(y)\in U_2$, then $\ran(y^{-1}) = \src(y)\in U_2$ and $\src(y^{-1}) = \ran(y) = \src(x)\in U_1$, which gives
\[
\src(xy) = \src(y) = \ran(y^{-1})\in U_{21} \subseteq W.
\] 
So we may assume that $\src(y)\in U_1$. Then, since $\src(x)\in U_{12}$, there exists $z\in \G$ such that $\src(z)\in U_2$ and $\ran(z) = \src(x)\in U_1$. We have $\ran(y^{-1}z) = \src(y)\in U_1$ and $\src(y^{-1}z)\in U_2$. This implies that $\src(y)\in U_{12} \subseteq W$, from which it follows that $xy\in \src^{-1}(W)$. 
\end{proof}

According to~\cite[Proposition 4.13]{steinberg0}, the center $\mathcal{Z}(A_R(\G))$ of $A_R(\G)$ is precisely the set of all ``class functions" on $\G$. Next, we define a more general version of that concept--relative each subset of $\G$, and give an alternative characterization of $\mathcal{Z}(A_R(\G))$ for prime $A_R(\G)$, as mentioned above.

\begin{definition}
Let $\G$ be a Hausdorff ample groupoid, let $R$ be a commutative unital ring, let $X\subseteq \G$, and let $f\in A_R(\G)$. We say that $f$ is a \textit{class function on $X$}, if the following hold for every $x\in X$:
\begin{enumerate}[\rm (1)]
\item $f(x)\neq 0$ implies that $\ran(x) = \src(x)$,
\item $f(zxz^{-1}) = f(x)$, for all $z\in \G$ such that $\src(x) = \ran(x) = \src(z)$ and $zxz^{-1}\in X$.
\end{enumerate}
\end{definition}

\begin{proposition}\label{centerclass}
Let $\G$ be a Hausdorff ample groupoid, and $R$ an integral domain, such that $A_R(\G)$ is prime. Suppose that $U_{ij}$ $(i,j \in \{1,2\})$ are as in Definition~\ref{partitiondef}, let $V := U_{11} \cup U_{22}$, let $W := U_{12} \cup U_{21}$, and let $f\in A_R(\G)$. Then $f \in \mathcal{Z}(A_R(\G))$ if and only if $f$ is a class function on $\src^{-1}(W)$.
\end{proposition}

\begin{proof}
Suppose that $f \in \mathcal{Z}(A_R(\G))$. Then $f$ is a class function (on $\G$), by~\cite[Proposition 4.13]{steinberg0}, and so, in particular, it is a class function on $\src^{-1}(W)$.

Conversely, suppose that $f$ is a class function on $\src^{-1}(W)$. We begin by taking arbitrary $g\in A_R(\G)$ and $x\in \src^{-1}(W)$, and showing that $(fg - gf)(x) = 0$. Given any $y \in \G$ such that $\src(x) = \src(y)$, we have $xy^{-1}\in \src^{-1}(W)$, by Lemma~\ref{INT_U}(2). Thus, by our hypothesis on $f$, we have $f(xy^{-1}) = 0$ whenever $\ran(xy^{-1})\neq \src(xy^{-1})$, for any such $y$. Then, using convolution gives 
\begin{equation}\label{classf1}
(fg)(x) = \sum_{y\in \src^{-1}\src(x)}f(xy^{-1})g(y) = \sum_{y\in (\src^{-1}\src(x))\cap (\ran^{-1}\ran(x))}f(xy^{-1})g(y),
\end{equation}
since $\ran(x) = \ran(xy^{-1})$ and $\src(xy^{-1}) = \ran(y)$. Now, $y^{-1}(xy^{-1})y\in \src^{-1}(W)$, for any $y \in \G$ such that $\src(x) = \src(y)$, and so
\[
f(xy^{-1}) = f(y^{-1}(xy^{-1})y) = f(y^{-1}x),
\]
by hypothesis on $f$. Letting $z = y^{-1}x$ (so that $\src(z) = \src(x)$ and $\ran(z) = \src(y)$), and using the fact that the ring $R$ is commutative, along with the hypothesis on $f$, we see that the last expression in $(\ref{classf1})$ equals
\[
\sum_{z\in (\src^{-1}\src(x))\cap (\ran^{-1}\ran(x))}g(xz^{-1})f(z) = \sum_{z\in \src^{-1}\src(x)}g(xz^{-1})f(z) = (gf)(x).\] 
Hence $(fg - gf)(x) = 0$, as claimed. Since $\G = \src^{-1}(W)\cup \src^{-1}(V)$, this means that $\supp(fg-gf)\subseteq \src^{-1}(V)$, for all $g\in A_R(\G)$.

Now, suppose that $fg-gf\neq 0$, for some $g\in A_R(\G)$. Then we can find a nonempty $B \in \G^{co}$ such that $B\subseteq \supp(fg-gf)$, since $\supp(fg-gf)$ is open. By the previous paragraph and the fact that $\src$ is open, this implies that $\src(B)$ is a nonempty open subset of $V$, contradicting Lemma~\ref{INT_U}(1). Thus, if $f$ is a class function on $\src^{-1}(W)$, then $fg-gf = 0$, for all $g\in A_R(\G)$, and so $f \in \mathcal{Z}(A_R(\G))$.
\end{proof}

\begin{remark} \label{decompremark}
Let $f$ be an arbitrary element of a Steinberg algebra $A_R(\G)$. Using the fact that $f$ is an $R$-linear combination of characteristic functions $\one_B$ ($B \in \G^{co}$), along with equations \eqref{setofeq1} and \eqref{setofeq2}, it is easy to see that $f = \one_{U}f\one_U$ for some compact slice $U\subseteq \G^{(0)}$. Moreover, in the notation of Definition~\ref{partitiondef}, 
\[
\one_U = \one_{U\cap U_1} + \one_{U \cap U_2} = \one_{UU_1} + \one_{UU_2},
\] 
and so
\[
f = \one_{U  U_1}f \one_{U  U_1} + \one_{U  U_2} f \one_{U  U_1} + \one_{U  U_1} f\one_{U  U_2} +\one_{U  U_2} f \one_{U  U_2} \in A_{11} + A_{21} + A_{12} + A_{22}.
\]
Clearly, each $A_{ij}$ is an $R$-submodule of $A_R(\G)$, and the above sum is direct. Hence 
\[
A_R(\G) = A_{11}  \oplus A_{21} \oplus A_{12} \oplus A_{22}
\] 
as $R$-modules. Moreover, for all $i,j,k,l \in \{1,2\}$, we have $A_{ij}A_{jk} \subseteq A_{ik}$, and $A_{ij}A_{kl} = \{0\}$ whenever $j \neq k$. So the above decomposition of $A$ can be viewed in matrix form:
\[
A_R(\G) = \begin{pmatrix} 
A_{11} & A_{12} \\ 
A_{21} & A_{22} 
\end{pmatrix}.
\]
Additionally, $A_{21}$ is a commutative $R$-subalgebra of $A_R(\G)$, since $A_{21}A_{21} = \{0\}$, and hence so is $\mathcal{Z}(A_R(\G)) + A_{21}$ (as well as $\mathcal{Z}(A_R(\G)) + A_{12}$). 
\end{remark}

Now we turn to our main goal--showing that the subalgebra $\mathcal{Z}(A_R(\G)) + A_{21}$ above is maximal for the property of being commutative, provided that $A_R(\G)$ is prime. We require a technical lemma.

\begin{lemma}\label{longlemma} 
Let $\G$ be a Hausdorff ample groupoid, and $R$ an integral domain, such that $A_R(\G)$ is prime. Suppose that $U_{i}$, $U_{ij}$, and $A_{ij}$ $(i,j \in \{1,2\})$ are as in Definition~\ref{partitiondef}. Let $f \in C(A_{21})$, where $C(A_{21})$ is the centralizer of $A_{21}$ in $A_R(\G)$, and write
\[
f = \begin{pmatrix} 
f_{11} & f_{12} \\ 
f_{21} & f_{22} \end{pmatrix} \in
\begin{pmatrix} 
A_{11} & A_{12} \\ 
A_{21} & A_{22} 
\end{pmatrix},
\]
using the decomposition of $A_R(\G)$ from Remark~\ref{decompremark}. Then the following hold.
\begin{enumerate}[\rm (1)]
\item $f_{12} = 0$.

\item If $x\in \G$ is such that $f_{11}(x)\neq 0$ or $f_{22}(x)\neq 0$, then $\src(x) = \ran(x)$.
	
\item Let $i,j \in \{1,2\}$ be distinct, and let $x \in U_i\G U_i$ be such that $\src(x) = \ran(x)$. Then $f_{ii}(x) = f_{jj}(zxz^{-1})$, for all $z \in U_j\G U_i$ satisfying $\src(z) = \ran(x)$.	

\item Let $i,j \in \{1,2\}$ be distinct, and let $x\in U_i\G U_i$ be such that $\src(x) = \ran(x)$ and $\src(y) = \ran(x)$ for some $y\in U_j\G U_i$. Then 
\[
(f_{11}+f_{22})(x) = (f_{11}+f_{22})(zxz^{-1}),
\]
for all $z\in \G$ satisfying $\src(z) = \ran(x)$.

\item If $f_{11} + f_{22} \neq 0$, then $W \subseteq \src(\supp(f_{11} + f_{22}))$, where $W := U_{12} \cup U_{21}$.

\item If $\G^{(0)}$ is not compact, then $f_{11} + f_{22} = 0$.
\end{enumerate}
\end{lemma}

\begin{proof}
First we note that 
\begin{equation} \label{matrix equation}
\begin{pmatrix} 
0 & 0 \\ 
g_{21}f_{11} & g_{21} f_{12} 
\end{pmatrix} 
= gf = fg = 
\begin{pmatrix} 
f_{12}g_{21} & 0 \\ 
f_{22}g_{21} & 0 
\end{pmatrix}
\end{equation}
for any element
\[
g = \begin{pmatrix} 
0 & 0 \\ 
g_{21} & 0 \end{pmatrix}
\]
of $A_{21}$.

\smallskip

(1) Since $\supp(f_{12}) \subseteq U_1\G U_2$, it suffices to take an arbitrary $x \in U_1\G U_2$, and show that $f_{12} (x) = 0$. Since $U_1\G U_2 = \ran^{-1}(U_1) \cap \src^{-1}(U_2)$ is open, we can find a $B \in \G^{co}$ such that $x \in B \subseteq U_1\G U_2$. Then $B^{-1} \subseteq U_2\G U_1$, and so $\one_{B^{-1}} \in A_{21}$, which implies that $f_{12}\one_{B^{-1}} = 0$, by equation \eqref{matrix equation}. Finally, using equations \eqref{setofeq1} and \eqref{setofeq2}, we have 
\[
f_{12} (x) = f_{12}(xx^{-1}x) = \one_{B^{-1}}f_{12}\one_{B^{-1}} (x^{-1}) = 0,
\]
as desired.

\smallskip

(2) Suppose that $x\in \G$ is such that $f_{11}(x)\neq 0$, but $\src(x) \neq \ran(x)$. Since $\supp(f_{11})\subseteq U_1\G U_1$, we have $\src(x), \ran(x) \in U_1$. Since $\G$ is Hausdorff and $U_1$ is open, we can find disjoint $B, B' \in \G^{co}$ such that $\src(x)\in B \subseteq U_1$ and $\ran(x)\in B' \subseteq U_1$. Then $\one_{B'}\one_{B} = 0$, while using equations \eqref{setofeq1} and \eqref{setofeq2}, together with the fact that $B, B' \subseteq \G^{(0)}$, gives
\[
\one_{B'}f_{11}\one_{B}(x) = f_{11}(\ran(x)^{-1}x\src(x)^{-1}) = f_{11}(\ran(x)x\src(x)) = f_{11}(x) \neq 0.
\]

Now let $D \in \G^{co} \setminus \{\emptyset\}$ be such that $D\subseteq U_2$. Since $A_R(\G)$ is assumed to be prime, and since $\one_{D} \neq 0 \neq \one_{B'}f_{11}\one_{B}$, we can find $h\in A_R(\G)$ such that  $\one_{D}h(\one_{B'}f_{11}\one_{B})\neq 0$. Since $\one_{D}h\one_{B'}\in A_{21}$, equation \eqref{matrix equation} implies that
\[
\one_{D}h(\one_{B'}f_{11}\one_{B}) = (\one_{D}h\one_{B'})f_{11}\one_{B} = f_{22}(\one_{D}h\one_{B'})\one_{B} = (f_{22}\one_{D}h)(\one_{B'}\one_{B}) = 0,
\]
producing a contradiction. Thus $\src(x) = \ran(x)$.

The case where $f_{22}(x)\neq 0$ can be handled analogously.

\smallskip

(3) We may assume that $i=1$ and $j=2$, since the opposite case can be proved analogously. Let $z \in U_2\G U_1$ be such that $\src(z) = \ran(x) \ (=\src(x))$. Then we can find $B \in \G^{co}$ such that $z\in B \subseteq U_2 \G U_1$. Since $\one_{B}\in A_{21}$, we have $\one_{B}f_{11} = f_{22}\one_{B}$, by equation \eqref{matrix equation}. Thus, using equations \eqref{setofeq1} and \eqref{setofeq2},
\[
f_{22}(zxz^{-1}) = \one_{B^{-1}} f_{22} \one_B(x) = \one_{B^{-1}} \one_B f_{11}(x) = \one_{\src(B)} f_{11}(x) = f_{11}(x),
\]
since $\ran(x) = \src(z) \in \src(B)$.

\smallskip

(4) Let $z\in \G$ be such that $\src(z) = \ran(x)$. If $\ran(z)\in U_j$, then, by $(3)$,
\[
(f_{11}+f_{22})(x) = f_{ii}(x) = f_{jj}(zxz^{-1}) = (f_{11}+f_{22})(zxz^{-1}).
\]
So let us suppose that $\ran(z)\in U_i$. Then $zxz^{-1} \in U_i\G U_i$, and $yz^{-1} \in U_j\G U_i$. Hence, using $(3)$ again, we have
\[
f_{ii}(zxz^{-1}) = f_{jj}(yz^{-1}zxz^{-1}zy^{-1}) = f_{jj}(yxy^{-1}) = f_{ii}(x),
\]
and so
$(f_{11}+f_{22})(x) = (f_{11}+f_{22})(zxz^{-1})$.

\smallskip

(5) Since $\supp(f_{11}+f_{22})$ is an open set, 
\[
\supp(f_{11}+f_{22}) = B_1\cup \dots \cup B_k,
\] 
for some $B_1, \dots, B_k \in \G^{co}$. Then 
\[
\src(\supp(f_{11}+f_{22})) = \src(B_1)\cup\ldots\cup \src(B_k) \subseteq \G^{(0)}.
\]
Since each $B_i$ is compact, and since $\src$ is continuous, each $\src(B_i)$ is compact as well. Given that $\G^{(0)}$ is Hausdorff, each $\src(B_i)$ must then be closed. Since $W$ is open (as a result of $\src$ being continuous and $\ran$ being open), 
\[
X := W \setminus \src(\supp(f_{11} + f_{21}))
\] 
is open as well. Let us assume that $f_{11} + f_{22} \neq 0$, and show that $X = \emptyset$.

Suppose that $X \neq \emptyset$. Then there exists $B \in \G^{co} \setminus \{\emptyset\}$ such that $B \subseteq X$. Since $A_R(\G)$ is prime, $\G$ is topologically transitive, by~\cite[Proposition 4.3]{steinberg1}. Given that $\supp(f_{11} + f_{21})$ is a nonempty open subset of $\G^{(0)}$, we can then find $y \in \G$ such that $\src(y) \in B$ and $\ran(y) \in \src(\supp(f_{11} + f_{22}))$. Thus $\ran(y)=\src(x)$, for some $x \in \supp(f_{11} + f_{22})$, where $\src(x) = \ran(x)$, by (2). We may assume that $x\in U_1\G U_1$, since the case where $x\in U_2\G U_2$ can be dispatched analogously.

Now, suppose that $\src(y)\in B\cap U_2$. Then $y^{-1}\in U_2\G U_1$ and $\src(y^{-1}) = \ran(x)$, and so $y^{-1}xy\in \supp(f_{11}+f_{22})$, by (4). But then $\src(y)\in \src(\supp(f_{11}+f_{22}))$, which contradicts $\src(y)\in X$. So we may assume that $\src(y)\in B\cap U_1$. Since $\src(y)\in W$, we must have $\src(y)\in U_{12}$. Hence there exists $z\in \G$ such that $\src(z)\in U_2$ and $\ran(z) = \src(y)$. Then $z^{-1}y^{-1}\in U_2\G U_1$, and so, by (4), we have 
\[
(f_{11}+f_{22})(x) = (f_{11}+f_{22})(y^{-1}xy).
\] 
But then, necessarily, $\src(y)\in \src(\supp(f_{11}+f_{22}))$, which again contradicts $\src(y)\in X$. Therefore $X = \emptyset$.

\smallskip

(6) Suppose that $\G^{(0)}$ is not compact, but $f_{11} + f_{22} \neq 0$. As shown in the proof of (5), $\src(\supp(f_{11} + f_{21} ))$ is a finite union of compact closed sets, and is therefore itself compact and closed. Since $\G^{(0)}$ is open but not compact, the set 
\[
X:= \G^{(0)} \setminus \src(\supp(f_{11} + f_{22}))
\] 
is open but necessarily nonempty. Then, by (5), 
\[
X \subseteq \G^{(0)} \setminus W = U_{11} \cup U_{22},
\]
which contradicts Lemma~\ref{INT_U}(1). Thus $f_{11} + f_{22} = 0$.
\end{proof}

It is shown in~\cite[Proposition 4.11]{steinberg0} (alternatively, see~\cite[Proposition 2.6]{rigby}) that $A_R(\G)$ is unital if and only if $\G^{(0)}$ is compact. A consequence of the previous lemma is that a prime non-unital Steinberg algebra $A_R(\G)$, with $\G^{(0)}$ having cardinality at least $2$ (i.e., with $\G$ not being a group), must have a trivial center.

\begin{proposition} \label{steinberg-center}
Let $\G$ be a Hausdorff ample groupoid with $|\G^{(0)}|\geq 2$, and let be $R$ an integral domain, such that $A_R(\G)$ is prime. If $\G^{(0)}$ is not compact, then $\mathcal{Z}(A_R(\G)) = 0$.
\end{proposition}

\begin{proof}
Suppose that $\G^{(0)}$ is not compact. Since $|\G^{(0)}|\geq 2$, we can decompose $\G^{(0)}$ into a union of disjoint sets, as in Definition~\ref{partitiondef}. Then $C(A_{21}) \subseteq A_{21}$, by Lemma~\ref{longlemma}(1,6), and so $\mathcal{Z}(A_R(G)) \subseteq A_{21}$. But, by Lemma~\ref{longlemma}(1), together with the analogue of that statement for $A_{12}$, we have $\mathcal{Z}(A_R(G)) \subseteq A_{11} \oplus A_{22}$. It follows that $\mathcal{Z}(A_R(G)) = 0$.
\end{proof}

We are now ready for our main result.

\begin{theorem} \label{steinbergtheorem}
Let $\G$ be a Hausdorff ample groupoid, and $R$ an integral domain, such that $A_R(\G)$ is prime. Suppose that $\G^{(0)} = U_1 \cup U_2$, for some disjoint open $U_1, U_2 \subseteq \G^{(0)}$, and let 
\[
A_{ij}  := \{f \in A_R(\G) \mid \supp(f) \subseteq U_i\G U_j\},
\]
for all $i,j \in \{1,2\}$. Then $\mathcal{Z}(A_R(G)) + A_{21}$ is a maximal commutative subalgebra of $A_R(\G)$. 
\end{theorem}

\begin{proof}
By Remark~\ref{decompremark}, $T:=\mathcal{Z}(A_R(G)) + A_{21}$ is a commutative subalgebra of $A_R(\G)$. Thus it suffices to show that $T$ is its own centralizer $C(T)$ in $A_R(\G)$.

Suppose that $f \in C(T)$. Then, in particular, $f \in C(A_{21})$. Using the notation of Lemma~\ref{longlemma}, we have
\[
f = \begin{pmatrix} 
f_{11} & 0 \\ 
f_{21} & f_{22} \end{pmatrix} \in
\begin{pmatrix} 
A_{11} & A_{12} \\ 
A_{21} & A_{22} 
\end{pmatrix},
\]
by (1) in that lemma. If $f_{11}+f_{22} = 0$, then $f \in A_{21} \subseteq T$. So let us assume that $f_{11}+f_{22} \neq 0$.

Suppose that $x \in \src^{-1}(W) \cap \supp(f_{11} + f_{22})$, where $W := U_{12} \cup U_{21}$, using the notation of Definition~\ref{partitiondef}. Then $\src(x) = \ran(x)$, by Lemma~\ref{longlemma}(2). Moreover, $\src(x) \in U_{ij}$, for some distinct $i,j \in \{1,2\}$, which means that $x \in U_i\G U_i$, and there exists $z \in \G$ such that $\src(z) \in U_j$ and $\ran(z) = \ran(x)$. Then $z^{-1} \in U_j\G U_i$ and $\src(z^{-1}) = \ran(x)$, and so $f_{11}+f_{22}$ is a class function on $\src^{-1}(W)$, by Lemma~\ref{longlemma}(4). Hence $f_{11}+f_{22} \in \mathcal{Z}(A_R(G))$, by Proposition~\ref{centerclass}, and therefore $f \in T$. Thus $C(T) = T$, as claimed.
\end{proof}

\section{Leavitt path algebras}

In this section we use Theorem~\ref{steinbergtheorem} to construct maximal commutative subalgebras of an arbitrary prime Leavitt path algebra--a special type of Steinberg algebra, but described in the graph-theoretic language associated with Leavitt path algebras. We begin by recalling relevant definitions and results. We refer the reader to~\cite{LPAbook} for general information about Leavitt path algebras, and to~\cite{rigby} for details on realizing Leavitt path algebras as Steinberg algebras. 

\subsection{Graphs}

A \textit{directed graph} $E=(E^{0},E^{1},\ran,\so)$ consists of two disjoint sets $E^{0}$ and $E^{1}$, where $E^0 \neq \emptyset$, together with functions $\ran,\so:E^{1}\rightarrow E^{0}$, called \emph{range} and \emph{source}, respectively. The elements of $E^{0}$ and $E^{1}$ are called \textit{vertices} and \textit{edges}, respectively. We shall refer to directed graphs as simply ``graphs" from now on.

Let $E=(E^{0},E^{1},\ran,\so)$ be a graph. A vertex $v \in E^0$ is called a \emph{sink} if $\so^{-1}(v) = \emptyset$, an \emph{infinite emitter} if $\so^{-1}(v)$ is infinite, \emph{regular} if it is neither a sink nor an infinite emitter, and \emph{singular} if it is not regular. A sequence $\alpha = e_1e_2\cdots$ of edges $e_1,e_2, \dots \in E^1$, such that $\ran(e_i)=\so(e_{i+1})$ for each $i$, is called a \emph{(finite) path}, in case it is finite, and an \emph{infinite path}, in case it is infinite. For any path $\alpha = e_1\cdots e_n$ we define $\so(\alpha):=\so(e_1)$ to be the \emph{source} of $\alpha$, $\ran(\alpha):=\ran(e_n)$ to be the \emph{range} of $\alpha$, and $|\alpha|:=n$ to be the \emph{length} of $\alpha$. If $\alpha = e_1e_2\cdots$ is an infinite path, we likewise define $\so(\alpha):=\so(e_1)$ to be the \emph{source} of $\alpha$. We view the elements of $E^0$ as paths of length $0$ (extending $\so$ and $\ran$ to $E^0$ via $\so(v) = v = \ran(v)$, for all $v\in E^0$), and denote by $\pth(E)$ the set of all paths in $E$. We denote the set of all infinite paths in $E$ by $\pth_{\infty}(E)$. If $u,v \in E^0$, and there exists $\alpha \in \pth(E)$ satisfying $\so(\alpha) = u$ and $\ran(\alpha)=v$, then we write $u \geq v$. We say that $E$ is \textit{downward directed} if for all $u,v\in E^0$, there exists $w\in E^0$ such that $u\geq w$ and $v\geq w$. 

\subsection{Leavitt path algebras}

Given a graph $E$ and a commutative ring $R$, the \textit{Leavitt path $R$-algebra} $L_R(E)$ \textit{of $E$} is the $R$-algebra generated by 
\[
\{v \mid v\in E^{0}\} \cup \{e,e^* \mid e\in E^{1}\},
\] 
subject to the following relations (where $\delta$ denotes the Kronecker delta):

{(V)} \ \ \ \  $vw = \delta_{v,w}v$ for all $v,w\in E^{0}$,\\
{(E1)} \ \ \ $\so(e)e=e\ran(e)=e$ for all $e\in E^{1}$,\\
{(E2)} \ \ \ $\ran(e)e^*=e^*\so(e)=e^*$ for all $e\in E^{1}$,\\
{(CK1)} \ $e^*f=\delta _{e,f}\ran(e)$ for all $e,f\in E^{1}$,\\
{(CK2)} \  $v=\sum_{e\in \so^{-1}(v)} ee^*$ for all regular $v\in E^{0}$.   

For all $v \in E^0$, we define $v^*:=v$, and for all paths $\alpha  = e_1 \cdots e_n$ ($e_1, \dots, e_n \in E^1$), we set $\alpha^*:=e_n^* \cdots e_1^*$, $\ran(\alpha^*):=\so(\alpha)$, and $\so(\alpha^*):=\ran(\alpha)$. It is easy to see that every element of $L_R(E)$ can be expressed in the form $\sum_{i=1}^n t_i\alpha_i\beta_i^*$, for some $t_i \in R$ and $\alpha_i,\beta_i \in \pth(E)$.

\subsection{Leavitt path algebras as Steinberg algebras}\label{LPASteinberg}

Given a graph $E$, the set of \emph{boundary paths} is defined as
\[\partial E := \pth_{\infty}(E) \cup \{\gamma \in \pth(E) \mid \ran(\gamma) \text{ is singular}\}.\]
We let
\[\G_E := \{(\alpha \gamma, |\alpha|-|\beta|, \beta \gamma) \mid \alpha, \beta \in \pth(E), \gamma \in \partial E, \text{ and } \ran(\alpha) = \ran(\beta) = \so(\gamma)\}.\]
Given $(\alpha,k,\beta),(\gamma,l,\delta) \in \G_E$, we define $(\alpha,k,\beta)^{-1} := (\beta,-k,\alpha)$ and
\[
(\alpha,k,\beta)(\gamma,l,\delta) := (\alpha,k+l,\delta),
\] 
in case $\beta=\gamma$. It is easy to see that, with these operations, $\G_E$ is a groupoid, where 
\[\G^{(0)}_E = \{(\gamma,0,\gamma) \mid \gamma \in \partial E\},\]
and each $(\alpha,k,\beta) \in \G_E$ is viewed as a morphism with domain $\beta$ and range $\alpha$.

To define a topology on $\G_E$, for all $\alpha,\beta \in \pth(E)$ with $\ran(\alpha) = \ran(\beta)$, and all finite $F \subseteq \so^{-1}(\ran(\alpha))$, let 
\[
Z(\alpha,\beta) := \{(\alpha \gamma,|\alpha|-|\beta|,\beta \gamma) \mid \gamma \in \partial E \text{ and } \ran(\alpha) = \ran(\beta) = \so(\gamma)\}
\]
and
\[
Z(\alpha,\beta,F) := Z(\alpha,\beta) \setminus \bigcup_{\gamma \in F} Z (\alpha\gamma, \beta\gamma).
\]
The sets $Z(\alpha,\beta,F)$ form a base of compact slices for a topology, under which $\G_E$ is a Hausdorff ample groupoid~\cite[Theorem 2.4]{rigby}.

Now, given a graph $E$ and a commutative unital ring $R$, define 
\[
\pi_E : L_R(E) \to A_R(\G_E)
\] 
via 
\begin{equation}\label{LPAiso}
\alpha\beta^* - \sum_{\gamma \in F} \alpha\gamma \gamma^* \beta^* \mapsto \one_{Z(\alpha,\beta, F)},
\end{equation}
for all $\alpha,\beta \in \pth(E)$ with $\ran(\alpha) = \ran(\beta)$, and all finite $F \subseteq \so^{-1}(\ran(\alpha))$. Then $\pi_E$ extends to an $R$-algebra isomorphism~\cite[Theorem 2.7]{rigby}.

\subsection{Maximal commutative subalgebras of Leavitt path algebras}

We are now ready to construct maximal commutative subalgebras of an arbitrary prime Leavitt path algebra. We note that a description of the center $\mathcal{Z}(L_R(E))$ of a Leavitt path algebra $L_R(E)$, which appears in the construction, can be found in~\cite{CenterSteinbergLeavitt}. 

\begin{theorem} \label{LPAtheorem}
Let $R$ be an integral domain, and $E$ a downward directed graph. Suppose that $P_1, P_2 \subseteq \pth(E)$ are nonempty subsets satisfying the following conditions.
\begin{enumerate}[\rm (1)]
\item $P_1 \cap P_2 = \emptyset$.
\item For each $i \in \{1,2\}$, $\alpha \in P_i$, and $\gamma \in \pth(E)$ such that $\ran(\alpha) = \so(\gamma)$, we have $\alpha\gamma \in P_i$.
\item For each $\alpha \in \pth(E)$ such that $\ran(\alpha)$ is singular, $\alpha \in P_1 \cup P_2$.
\item For each $\gamma \in \pth_{\infty}(E)$, there exist $\alpha \in P_1 \cup P_2$ and $\delta \in \pth_{\infty}(E)$ such that $\gamma=\alpha \delta$.
\end{enumerate}
Then
\[\mathcal{Z}(L_R(E)) + \langle \alpha\beta^* \mid \alpha \in P_1, \beta \in P_2 \rangle\]
is a maximal commutative subalgebra of $L_R(E)$, where the second summand is the $R$-subalgebra of $L_R(E)$ generated by the relevant $\alpha\beta^*$.
\end{theorem}

\begin{proof}
For each $i \in \{1,2\}$, let $U_i = \bigcup_{\alpha \in P_i} Z(\alpha, \alpha) \subseteq \G_{E}^{(0)}$. As unions of basic open sets, $U_1$ and $U_2$ are open. Note that for all $\alpha, \beta \in \pth(E)$, we have 
\[
Z(\alpha, \alpha) \cap Z(\beta, \beta) \neq \emptyset
\] 
if and only if there exists $\gamma \in \pth(E)$, such that $\alpha = \beta\gamma$ or $\beta = \alpha\gamma$. Thus, conditions (1) and (2) imply that $U_1 \cap U_2 = \emptyset$. On the other hand, conditions (3) and (4) imply that $\G_{E}^{(0)} = U_1 \cup U_2$. Now let 
\[
{A_{12}} := \{f \in A_R(\G_{E}) \mid \supp(f) \subseteq U_1\G_{E} U_2 \}.
\]
Our hypotheses on $R$ and $E$ imply that $A_R(\G_{E}) \cong L_R(E)$ is prime, by \cite[Theorem 5.3]{steinberg1}. 
Therefore, by Theorem~\ref{steinbergtheorem}, $\mathcal{Z}(A_R(\G_{E})) + A_{12}$ is a maximal commutative subalgebra of $A_R(\G_{E})$. Let $\pi_E : L_R(E) \to A_R(\G_E)$ be the isomorphism described in~\eqref{LPAiso}. Then $\pi_E (\mathcal{Z}(L_R(E))) = \mathcal{Z}(A_R(\G_{E}))$, and so to conclude the proof it suffices to show that $\pi_E(S) = A_{12}$, where 
\[
S := \langle \alpha\beta^* \mid \alpha \in P_1, \beta \in P_2 \rangle.
\]

We note that 
\[
(\alpha \gamma, |\alpha|-|\beta|, \beta \gamma) \in (\delta, 0, \delta)\G_{E}(\zeta, 0, \zeta)
\]
if and only if $\alpha\gamma = \delta$ and $\beta \gamma = \zeta$, for all $\alpha, \beta \in \pth(E)$ and $\gamma, \delta, \zeta \in \partial E$, such that $\ran(\alpha) = \ran(\beta) = \so(\gamma)$. It follows that, for all $\alpha, \beta \in \pth(E)$ with $\ran(\alpha) = \ran(\beta)$, we have $Z(\alpha,\beta) \subseteq U_1\G_{E} U_2$ if and only if $\alpha \in P_1$ and $\beta \in P_2$. Thus, in particular, $\pi_E(\alpha\beta^*) = \one_{Z(\alpha, \beta)} \in A_{12}$, for all $\alpha \in P_1$ and $\beta \in P_2$, which implies that $\pi_E(S) \subseteq A_{12}$, since $A_{12}$ is an $R$-subalgebra of $A_R(\G_{E})$ (see Remark~\ref{decompremark}). For the opposite inclusion, we note that every element of $A_{12}$ is of the form $\sum_{i=1}^nt_i\one_{B_i}$, for some $t_i \in R$ and $B_i \in \G_{E}^{co}$, such that $B_i \subseteq U_1\G_{E} U_2$. Thus, it suffices to take an arbitrary $B \in \G_{E}^{co}$ satisfying $B \subseteq U_1\G_{E} U_2$, and show that $\one_{B} \in \pi_E(S)$.

Using the standard base for $\G_{E}$ described in Section~\ref{LPASteinberg}, and the fact that $B$ is compact, we have 
\[
B = \bigcup_{i=1}^n Z(\alpha_i,\beta_i,F_i) = \bigcup_{i=1}^n \bigg(Z(\alpha_i,\beta_i) \setminus \bigcup_{\gamma \in F_i} Z(\alpha_i\gamma,\beta_i\gamma)\bigg),
\]
for some appropriate $\alpha_i, \beta_i \in \pth(E)$ and finite $F_i \subseteq \so^{-1}(\ran(\alpha_i))$. Given that $Z(\alpha_i,\beta_i,F_i) \subseteq U_1\G_{E} U_2$, necessarily $\alpha_i \in P_1$ and $\beta_i \in P_2$, for each $i$, as discussed above. Upon taking intersections of the sets $F_i$, as needed, we may assume that the path pairs $(\alpha_i, \beta_i)$ are pairwise distinct.

Now, suppose that 
\[
Z(\alpha_j,\beta_j,F_j) \cap Z(\alpha_k,\beta_k,F_k) \neq \emptyset,
\] 
for some $j,k \in \{1, \dots, n\}$ distinct. Then, upon interchanging the indices if necessary, we must have $\alpha_j = \alpha_k\delta$ and $\beta_j = \beta_k\delta$ for some $\delta \in \so^{-1}(\ran(\alpha_k)) \setminus E^0$. But then letting
\[
F_j' = \{\zeta \in F_j \mid \delta\zeta = \eta\theta \text{ for some } \eta \in F_k, \theta \in \pth(E)\},
\]
we have
\[
Z(\alpha_j,\beta_j,F_j) \cup  Z(\alpha_k,\beta_k,F_k) = Z(\alpha_j,\beta_j, F_j') \cup Z(\alpha_k,\beta_k,F_k \cup \{\delta\}),
\]
and the latter two sets are disjoint. So, upon making such a substitution, we may assume that 
\[
Z(\alpha_j,\beta_j,F_j) \cap Z(\alpha_k,\beta_k,F_k) = \emptyset.
\] 
Using a recursive version of this process, we may assume that the sets $Z(\alpha_i,\beta_i,F_i)$ are pairwise disjoint. Then $\one_{B} = \sum_{i=1}^n \one_{Z(\alpha_i,\beta_i,F_i)}$, and so
\[
\one_{B} = \sum_{i=1}^n \pi_E \bigg( \alpha_i\beta_i^* - \sum_{\gamma \in F_i}\alpha_i\gamma\gamma^*\beta_i^*\bigg) \in \pi_E(S),
\]
as claimed.
\end{proof}

From this result we can deduce a more general version of~\cite[Theorem 3.2]{BvWZ_prime}, which pertains to prime Leavitt path algebras over fields.

\begin{corollary} \label{LPA-cor}
Let $R$ be an integral domain, and $E$ a downward directed graph. Suppose that $E^0 = V_1 \cup V_2$, where the union is disjoint, and let
\[
L_R(V_1,V_2) = \langle \alpha\beta^* \mid \alpha, \beta \in \pth(E) \text{ such that } \so(\alpha) \in V_1, \so(\beta) \in V_2 \rangle.
\]
Then $\mathcal{Z}(L_R(E)) + L_R(V_1,V_2)$ is a maximal commutative subalgebra of $L_R(E)$.
\end{corollary}

\begin{proof}
For each $i \in \{1,2\}$, let 
\[
P_i = \{\alpha \in \pth(E) \mid \so(\alpha) \in V_i\}.
\] 
Then $P_1$ and $P_2$ clearly satisfy conditions (1)--(4) in Theorem~\ref{LPAtheorem}. Since 
\[
L_R(V_1,V_2) = \langle \alpha\beta^* \mid \alpha \in P_1, \beta \in P_2 \rangle,
\]
the desired conclusion follows from the theorem.
\end{proof}

The above corollary and~\cite[Theorem 3.2]{BvWZ_prime} produce maximal commutative subalgebras of $L_R(E)$ only for downward directed graphs $E$ having at least two vertices. In contrast to this, Theorem~\ref{LPAtheorem} also applies to downward directed graphs with just one vertex, as illustrated in the example below, provided that there are at least two distinct edges. If, however, $E$ is a graph with only one vertex and at most one edge, then $L_R(E)$ is commutative.

\begin{example}
Let $E$ be the graph where $E^0 = \{v\}$ and $E^1 = \{e_1,e_2\}$, pictured below.
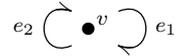
\[
\xymatrix{{\bullet}^{v} \ar@(ur,dr)^{e_1} \ar@(dl,ul)^{e_2}}
\]
For each $i \in \{1,2\}$, let
\[
P_i = \{\alpha \in \pth(E) \mid \alpha = e_i\beta \text{ for some } \beta \in \pth(E)\}.
\]
Then $P_1$ and $P_2$ certainly satisfy conditions (1)--(4) in Theorem~\ref{LPAtheorem}. So if $R$ is any integral domain, then
\[
\mathcal{Z}(L_R(E)) + \langle \alpha\beta^* \mid \alpha \in P_1, \beta \in P_2 \rangle = Rv + \langle \alpha\beta^* \mid \alpha \in P_1, \beta \in P_2 \rangle
\]
\[
= Rv + \langle e_1\alpha\beta^*e_2^* \mid \alpha, \beta \in \pth(E)\rangle
\]
is a maximal commutative subalgebra of $L_R(E)$, where $\mathcal{Z}(L_R(E)) = Rv$, by~\cite[Theorem 3.3]{CenterSteinbergLeavitt}.
\end{example}

There are various generalizations of Leavitt path algebras, in the literature, that can also be realized as Steinberg algebras, such as ultragraph Leavitt path algebras~\cite{hazrat-nam} and Kumjian--Pask algebras (of finitely-aligned higher-rank graphs)~\cite{clark-pangalela}. So while Theorem~\ref{steinbergtheorem} can be specialized to these as well, we shall not provide the details here, since the definitions of the algebras in question are quite technical. We note, however, that~\cite[Theorem 4.6]{clark-canto-isfahani} constructs maximal commutative subalgebras of a Kumjian--Pask algebra (over a row-finite higher-rank graph with no sources).

\section{Questions for further research}

Our constructions above produce large classes of maximal commutative subalgebras of Steinberg and Leavitt path algebras. We are, however, still quite far from a classification of such subalgebras in general, or even in interesting special cases.

\begin{question}
Given an integral domain $R$, and reasonable hypotheses on the Hausdorff ample groupoid $\G$, respectively graph $E$, can one classify all the maximal commutative subalgebras of $A_R(\G)$, respectively $L_R(E)$?
\end{question}

Here is a possibly more manageable variant of this question, stated just for Steinberg algebras, though its analogue for Leavitt path algebras, in the notation of Theorem~\ref{LPAtheorem}, would likewise be worth exploring.

\begin{question} \label{ref-q}
Let $\G$ be a Hausdorff ample groupoid, and $R$ an integral domain, such that $A_R(\G)$ is prime. Suppose that $\G^{(0)} = U_1 \cup U_2$ and $\G^{(0)} = U_1' \cup U_2'$ are two different partitions of $\G^{(0)}$ into disjoint open subsets. Under what conditions on $U_1$, $U_2$, $U_1'$, and $U_2'$ are the resulting maximal commutative subalgebras of $A_R(\G)$, as constructed in Theorem~\ref{steinbergtheorem}, isomorphic?
\end{question}

\section*{Acknowledgements}

Anna Cichocka was supported by the Polish National Science Centre grant UMO-2021/41/N/ST1/03067. 

The authors would like to thank the referee for a careful reading of the manuscript and suggesting Question~\ref{ref-q}.


\begin{thebibliography}{99}

\bibitem{LPAbook} G.\ Abrams, P.\ Ara, and M.\ Siles Molina, \textit{Leavitt path algebras,}
Lecture Notes in Mathematics \textbf{2191}, Springer-Verlag, London, 2017.

\bibitem{BvWZ_prime} G.\ Bajor, L.\ van Wyk, and M.\ Ziembowski, \textit{Construction of a class of maximal commutative subalgebras of prime Leavitt path algebras}, Forum Math.\ \textbf{33} (2021) 1573--1590.

\bibitem{BCvWZ} G.\ Bajor, A.\ Cichocka, L.\ van Wyk, and M.\ Ziembowski, \textit{Maximal commutative subalgebras of Leavitt path algebras}, Commun.\ Contemp.\ Math.\ \textbf{26} (2024) 2250077 (28 pages).

\bibitem{bavula} V.\ V.\ Bavula, \textit{Structure of maximal commutative subalgebras of the first Weyl algebra}, Ann.\ Univ.\ Ferrara Sez.\ VII (N.S.) \textbf{51} (2005) 1--14.

\bibitem{dynkin} E.\ B.\ Dynkin, \textit{Semisimple subalgebras of semisimple Lie algebras} (in Russian), Mat.\ Sbornik N.\ S.\ \textbf{30} (1952) 349--462.

\bibitem{elduque} A.\ Elduque, \textit{On maximal subalgebras of central simple Mal’cev algebras}, J.\ Algebra \textbf{103} (1986) 216--227.

\bibitem{ELS} A.\ Elduque, J.\ Laliena, and S.\ Sacrist\'an, \textit{Maximal subalgebras of associative superalgebras}, J.\ Algebra \textbf{275} (2004) 40--58.

\bibitem{Hazrat-Li} R.\ Hazrat and H. Li, \textit{A note on the centralizer of a subalgebra of the Steinberg algebra}, St.\ Petersburg Math.\ J.\ \textbf{33} (2022) 179--184.

\bibitem{hazrat-nam} R.\ Hazrat and T.\ G.\ Nam, \textit{Realizing ultragraph Leavitt path algebras as Steinberg algebras}, J.\ Pure Appl.\ Algebra \textbf{227} (2023) 107275 (20 pages).

\bibitem{jacobson} N.\ Jacobson, \textit{Schur’s theorems on commutative matrices}, Bull.\ Amer.\ Math.\ Soc.\ \textbf{50} (1944) 431--436.

\bibitem{KW} T.\ Kajiwara and Y.\ Watatani, \textit{Maximal abelian subalgebras of $C^*$ -algebras associated with complex dynamical systems and self-similar maps}, J.\ Math.\ Anal.\ Appl.\ \textbf{455} (2017) 1383--1400.

\bibitem{karamzadeh} N.\ S.\ Karamzadeh, \textit{Maximal abelian subalgebras of $C^*$-algebras associated with complex dynamical systems and self-similar maps}, Math.\ Inequal.\ Appl.\ \textbf{13} (2010) 625--628.

\bibitem{malcev} A.\ I.\ Malcev, \textit{Commutative subalgebras of semi-simple Lie algebras}, Amer.\ Math.\ Soc.\ Transl.\ \textbf{1951} (1951) 1--15.

\bibitem{SteinbergIntro} L.\ Orloﬀ Clark, C.\ Farthing, A.\ Sims, and M.\ Tomforde, \textit{A groupoid generalisation of Leavitt path algebras}, Semigroup Forum \textbf{89} (2014) 501--517.

\bibitem{clark-canto-isfahani} L.\ Orloﬀ Clark, C.\ Gil Canto, and A.\ Nasr-Isfahani, \textit{The cycline subalgebra of the Kumjian-Pask algebra}, Proc.\ Amer.\ Math.\ Soc.\ \textbf{145} (2017) 1969--1980.

\bibitem{lisarooz} L.\ Orloﬀ Clark and R.\ Hazrat, \textit{\'Etale groupoids and Steinberg algebras, a concise introduction}, in \textit{Leavitt path algebras and classical K-theory,} Indian Statistical Institute Series, Springer, Singapore, 2020, pp.\ 73--101.

\bibitem{CenterSteinbergLeavitt} L.\ Orloﬀ Clark, D.\ Martín Barquero, C.\ Martín González,  and M.\ Siles Molina, \textit{Using the Steinberg algebra model to determine the center of any Leavitt path algebra}, Israel J.\ Math.\ \textbf{230} (2019) 23--44. 

\bibitem{clark-pangalela} L.\ Orloﬀ Clark and Y.\ E.\ P.\ Pangalela, \textit{Kumjian–Pask algebras of finitely aligned higher-rank graphs}, J.\ Algebra \textbf{482} (2017) 364--397.

\bibitem{renault} J.\ Renault, \textit{A groupoid approach to $C^*$-algebras}, Lecture Notes in Mathematics \textbf{793}, Springer, Berlin, 1980.

\bibitem{rigby} S.\ W.\ Rigby, \textit{The groupoid approach to Leavitt path algebras}, in \textit{Leavitt path algebras and classical K-theory,} Indian Statistical Institute Series, Springer, Singapore, 2020, pp.\ 21--72.

\bibitem{schur} J.\ Schur, \textit{Zur Theorie der vertauschbaren Matrizen}, J.\ Reine Angew.\ Math.\ \textbf{130} (1905) 66--76.

\bibitem{steinberg0} B.\ Steinberg, \textit{A groupoid approach to discrete inverse semigroup algebras}, Adv.\ Math.\ \textbf{223} (2010) 689--727.

\bibitem{steinberg1} B. Steinberg, \textit{Prime \'etale groupoid algebras with applications to inverse semigroup and Leavitt path algebras}, J.\ Pure Appl.\ Algebra \textbf{223} (2019)  2474--2488. 

\bibitem{SBWZ} J.\ Szigeti, J.\ van den Berg, L.\ van Wyk, and M.\ Ziembowski, \textit{The maximum dimension of a Lie nilpotent subalgebra of $\mathbb{M}_n(F)$ of index $m$}, Trans.\ Amer.\ Math.\ Soc.\ \textbf{372} (2019) 4553--4583.

\end{thebibliography}
\end{document}